\documentclass[12pt]{article}

\usepackage{latexsym}
\usepackage{amssymb}
\usepackage{verbatim}
\usepackage{amsthm}
\usepackage{amsmath}
\usepackage{yfonts}

\let\cal\mathcal

\let\cal=\mathcal      

\def\mcc{M\raise.5ex\hbox{c}C}
\def\mccarthy{M\raise.5ex\hbox{c}Carthy}

\def\m{Mult}

\def\a{\alpha}
\def\l{\lambda}

\def\vare{\varepsilon}

\let\i=\infty

\def\={\ = \ }

\def\A{{\cal A}}

\def\E{E_\l}

\def\C{\mathbb C}
\def\R{\mathbb R}

\def\D{\mathbb D}

\def\be{\setcounter{equation}{\value{theorem}} \begin{equation}}
\def\ee{\end{equation} \addtocounter{theorem}{1}}
\def\beq{\begin{eqnarray*}}
\def\eeq{\end{eqnarray*}}

\def\vs{\vskip 5pt}

\def\bp{{\sc Proof: }}
\def\ep{{}{\hfill $\Box$} \vskip 5pt \par}

\def\bl{\begin{lemma}}
\def\el{\end{lemma}}
\def\bt{\begin{theorem}}
\def\et{\end{theorem}}
\def\bprop{\begin{prop}}
\def\eprop{\end{prop}}
\def\bd{\begin{definition}}
\def\ed{\end{definition}}
\def\br{\begin{remark}}
\def\er{\end{remark}}
\def\bexer{\begin{exercise}}
\def\eexer{\end{exercise}}
\def\bfig{\begin{figure}}
\def\efig{\end{figure}}

\numberwithin{equation}{section}

\usepackage[usenames,dvipsnames] {color}

\date{February 19, 2014}
\title{The implicit function theorem and  free algebraic sets
\footnote{MSC 14M99, 16S50. Key Words: NC functions, free holomorphic functions, free algebraic sets  }
}
\author{Jim Agler
\thanks{Partially supported by National Science Foundation Grant
DMS 1068830}
\\ U.C. San Diego\\ La Jolla, CA 92093
\and
John E. M\raise.5ex\hbox{c}Carthy
\thanks{Partially supported by National Science Foundation Grant  
DMS 1300280
}
\\ Washington University\\ St. Louis, MO 63130}

\def\be{\begin{equation}}
\def\ee{\end{equation}}
\def\norm#1{\| #1 \|}

\def\m{\mathbb{M}}
\def\md{\mathbb{M}^{[d]}}

\def\mn{\mathbb{M}_n}
\def\mnd{\mathbb{M}_n^d}

\def\mnk{\mathbb{M}_n^k}

\def\c{\mathbb{C}}
\def\D{\mathbb{D}}
\def\r{\mathbb{R}}
\def\n{\mathbb{N}}

\def\k{\mathcal{K}}

\def\L{\mathcal{L}}

\def\O{\Omega}

\def\set#1#2{\{ #1 \, | \, #2\}}

\def\idd{{\rm id}}

\def\vspace{\mathcal{V}}

\def\pd{\mathbb{P}^d}
\newcommand\gdel{G_\delta}

\def\si{s^{-1}}

\def\l{\lambda}

\def\Ax{{\mathcal A}_x}

\renewcommand\={\ = \ }
\def\vs{\vskip 5pt}

\def\C{\mathbb{C}}

\def\i{\infty}
\renewcommand\d{\delta}

\def\E{{\mathcal{E}}}

\def\bd{\begin{defin}}
\def\ed{\end{defin}}
\def\be{\begin{equation}}
\def\ee{\end{equation}}
\def\bp{\begin{proof}}
\def\ep{\end{proof}}
\def\begm{\left(\begin{matrix}}
\def\endm{\end{matrix}\right)}
\def\begbm{\begin{bmatrix}}
\def\endbm{\end{bmatrix}}

\def\E2{E^{[2]}}

\theoremstyle{definition}
\newtheorem{theorem}[equation]{Theorem}
\newtheorem{defin}[equation]{Definition}
\newtheorem{lem}[equation]{Lemma}
\newtheorem{lemma}[equation]{Lemma}
\newtheorem{prop}[equation]{Proposition}

\newtheorem{ques}[equation]{Question}

\newtheorem{remark}[equation]{Remark}

\newtheorem{cor}[equation]{Corollary}
\newtheorem{example}[equation]{Example}

\def\bl{\begin{lem}}
\def\el{\end{lem}}
\def\bt{\begin{theorem}}
\def\et{\end{theorem}}

\def\U{\mathcal U}
\def\beginb{\begin{bmatrix}}
\def\endb{\end{bmatrix}}
\def\beginp{\begin{pmatrix}}
\def\endp{\end{pmatrix}}
\def\bkt{\ [\times]\ }
\def\ak{a^{(k)}}

\begin{document}

\bibliographystyle{plain}

\maketitle

Abstract:
We prove an implicit function theorem for non-commutative functions.
We use this to show that if $p(X,Y)$ is a generic non-commuting polynomial in two 
variables, and $X$ is a generic matrix, then all solutions $Y$ of $p(X,Y)=0$ will commute with $X$.

\section{Introduction}
\label{seca}

A free polynomial, or nc polynomial (nc stands for non-commutative), is a polynomial in non-commuting variables.
Let $\pd$ denote the algebra of free polynomials in $d$ variables. If $p \in \pd$, it makes sense to think of $p$ as a function that can be evaluated on matrices.
Let $\mn$ be the set of $n$-by-$n$ complex matrices,
and $\md = \cup_{n=1}^\i \mnd$. A free algebraic set is a subset of $\md$ that is the 
common zero set of a collection of free polynomials.

One  principal result in this paper
  is that, in some generic sense,
if $X$ and $Y$ are in $\mn$ and $p(X,Y) = 0$ for some $p \in {\mathbb P}^2$, then $Y$ commutes with $X$.
To explain what we mean by ``generically'', consider the following specific example.
Let $a,b,c$ be complex numbers, and
let 
$$p(X,Y) \= aX^2 + bXY + cYX . $$
Then we show in Proposition~\ref{propg3} that
if $p(X,Y) =0$, then $Y$ must commute with $X$
unless $bX$ and $-cX$ have a common eigenvalue.
We extend this to a general theorem about free algebraic sets defined by $d-1$ polynomials in $d$ variables in Theorem~\ref{thmg1}.

 An nc function is a  generalization of a free polynomial, just as a holomorphic function in scalar variables can be thought of as a
generalization of a polynomial in commuting variables.

To make this precise, 
define a  {\em graded function}  to be a function $f$, with domain 
some subset of $\md$, and
with the property that
if $x \in \mnd$, then $f(x) \in \mn$.
\begin{defin}\label{defa1}
An {\em nc-function} is a graded function $f$ defined on a
set $\Omega \subseteq \md$ such that

i) If $x,y, x\oplus y \in \Omega$, then $f(x \oplus y) = f(x) \oplus f(y)$.

ii) If $s \in \mn$ is invertible and $x, s^{-1} x s \in \Omega \cap \mnd$, then $f(\si x s) = \si f(x) s$.

\end{defin}

Free polynomials are examples of nc-functions.
Nc-functions have been studied for a variety of reasons:
by Anderson \cite{an69} as a generalization of the Weyl calculus;
 by Taylor \cite{tay73}, in the context of the functional calculus 
for non-commuting operators;
Popescu \cite{po06,po08,po10,po11}, in the context of extending classical function theory to 
$d$-tuples of bounded operators;
 Ball, Groenewald and Malakorn \cite{bgm06}, in the context of extending realization formulas
from functions of commuting operators to functions of non-commuting operators;
Alpay and Kalyuzhnyi-Verbovetzkii \cite{akv06} in the context of
realization formulas for rational functions that are $J$-unitary on the boundary of the domain;
Helton \cite{helt02} in proving  positive matrix-valued functions are sums of squares;
 and
Helton, Klep and McCullough \cite{hkm11a,hkm11b}
and Helton and McCullough \cite{hm12} in the context of developing a descriptive theory of the domains on
which LMI and semi-definite programming apply.
Recently, Kaliuzhnyi-Verbovetskyi  and Vinnikov have written a monograph on the subject
\cite{kvv12}.

We need to introduce topologies  on $\md$. 
First, we define the {\em  disjoint union topology} by saying
that  a set $U$ is open in the disjoint union topology if and only if $ U \cap \mnd$ is open for every $n$. We shall abbreviate disjoint union as d.u. A set $V \subset \md$ is {\em bounded} if there exists
a positive real number $B $
such that $\| x \| \leq B$ for every $x$ in $V$.

We shall say that a set $ \Omega \subseteq \md$ is an {\em nc domain} if it is closed under direct 
sums and unitary conjugations, and is open in the d.u. topology. We shall say that a topology is an {\em admissible topology}
if it has a basis of bounded nc domains.

\begin{defin}
\label{defa2}
Let $\tau$ be an admissible topology on $\md$, and let $\O$ be a $\tau$-open set.
A {$\tau$-holomorphic function} is an nc-function $f : \O \to \m$ that is $\tau$ locally bounded.
\end{defin}

Note that if $f$ is a $\tau$-holomorphic function, then for every $a \in \O \cap \mnd$ and every $h \in \mnd$,
the derivative
\be
\label{eqa1}
Df(a) [h] \ := \
\lim_{t \to 0} 
\frac{1}{t} [ f(a + th) - f(a) ]
\ee
exists \cite{amfree}.

In Section~\ref{secb1} we shall define some  particular admissible topologies: the fine, fat, and free
topologies.
The properties of nc holomorphic functions turn out to depend critically on the choice of topology.
In the free topology there is an Oka-Weil theorem, and in particular every free holomorphic function $f$
has the property that $f(x)$ is in the algebra generated by $x$ for every $x$ in the domain
\cite{amfree}; this property was crucial in the authors' study of Pick interpolation for free holomorphic functions \cite{amfreePick}.
Pointwise approximation of holomorphic functions by polynomials fails for the fine and fat  topologies: the following result is a consequence
of Theorem \ref{thmf1}.

\bt
\label{thma1}
For $d \geq 2$, there is a  fat holomorphic function that is not 
pointwise approximable by free polynomials.
\et

The fine and fat topologies do have good properties, though. J.~Pascoe proved
 an Inverse Function theorem for fine holomorphic maps \cite{pas13}. We extend this in 
Theorem~\ref{thme1} to the fat category. 
In Theorem~\ref{thmc1}, we prove 
an Implicit Function theorem in the fine and fat topologies.
Here is a special case, when the zero set is of a single function.

\vs

\bt
Let $U$ an nc domain.
Let $f$ be a fine (resp. fat) holomorphic function on $U$.
Suppose that
\[
\forall \, a \in U, \ \left[ \frac{\partial f}{ \partial x^d} (a) [h] = 0 \right] \ \Rightarrow h = 0.
\]
Let $W$ be the projection onto the first $d-1$ coordinates of $Z_f \cap U$.
Then there is a fine (resp. fat) holomorphic  function $g$ on $W$ such that
\[
Z_f \cap U \=
\{ (y,g(y)) \, : \, y \in W \} .
\]
\et

\vs

The advantage of working with the fat topology is that
we prove in Theorem~\ref{prope1} that if the derivative of a fat holomorphic function is full rank 
at a point, then it is full rank in a fat neighborhood of the point. This fact, along with the Implicit function theorem, is used to prove Theorem~\ref{thma1}.

\vs

Our final result is that there is no Goldilocks topology.
In Theorem~\ref{thmh1}
we show that if $\tau$ is an admissible topology on $\md$ with the properties that:

(i) free polynomials are continuous from $(\md, \tau) $ to $(\m^{[1]}, d.u.)$

 (ii) $\tau$-holomorphic functions are pointwise approximable by nc polynomials,

then there is no $\tau$ Implicit function theorem.

\section{Background material}
\label{secb}


The following lemma is in \cite{hkm11b} and \cite{kvv12}.
\begin{lem}\label{lembb1}(cf. Lemma 2.6 in \cite{hkm11b}). 
Let $\Omega$ be an nc set in $\m^d$, 
and let $f$ be an
 nc-function on $\Omega$. Fix $n \ge 1$ and $\Gamma \in \m_n$. If $a,b \in \Omega \cap \mn^d$ and
\[
\begin{bmatrix}b&b\Gamma -\Gamma a\\ 0&a\end{bmatrix} \in \Omega \cap \m_{2n}^d,
\]
then
\be\label{bb1}
f(\begin{bmatrix}b&b\Gamma -\Gamma a\\ 0&a\end{bmatrix}) \= \begin{bmatrix}f(b)&f(b)\Gamma-\Gamma f(a)\\ 0&f(a)\end{bmatrix}.
\ee
\end{lem}

If we let $b = a + th$ and $\Gamma = \frac{1}{t}$, and let $t$ tend to $0$, we get
\begin{lem}
\label{lembb2}
Let $U \subseteq \md $ be d.u. open, and suppose that $a \in U$ and 
$ \begin{bmatrix}a&h\\0&a\end{bmatrix}  \in U$. Then
\be
\label{eqbb2}
f( \begin{bmatrix}a&h\\0&a\end{bmatrix} )
\=
\begin{bmatrix}f(a)&Df(a)[h]\\0&f(a)\end{bmatrix}  .
\ee
\end{lem}

Combining these two results, we get
\begin{lem}\label{lembb3}
Let $\Omega$ be an nc domain in $\m^d$, 
 let $f$ be an
 nc-function on $\Omega$, and let $a \in \Omega$.
Then 
\[
Df(a) [ a \Gamma - \Gamma a] \= f(a) \Gamma - \Gamma f(a) .
\]
\end{lem}

By an $\L(\c^\ell, \c^k)$ valued nc function we mean a $k$-by-$\ell$ valued 
matrix of nc functions. 
An  $\L(\c,\c^k)$ valued nc function $f$ can be thought of as 
a vector of $k$ nc functions, $(f_1, \dots, f_k)^t$.
When $k=d$, we 
shall call a $d$-tuple of nc functions on a set in $\md$ an {\em nc map.}

If $\Phi$ is an $\L(\c^\ell, \c^k)$ valued nc function,
then if $a \in \mnd$, the derivative $D\Phi(a)$ is in $\L(\mnd, \mn \otimes \L(\c^\ell, \c^k))$.

\section{Admissible Topologies}
\label{secb1}

\subsection{The fine topology}

The {\em fine topology} is the topology that has as a basis all nc domains.
Since this is the largest admissible topology, for any admissible topology $\tau$,
any $\tau$-holomorphic function is automatically 
fine holomorphic. 

\begin{lem}
\label{lemb1}
Suppose $\Omega$ is an nc domain, and $f : \Omega \to \m$ is d.u. locally bounded.
Then $f$ is a fine holomorphic function.
\end{lem}
\bp
Let $a \in \Omega$, and $\| f(a) \| = M$.
Let $U = \{ x \in \Omega : \| f(x) \| < M + 1 \}$.
Then  $U$ is an nc set, and by  \cite{amfree} it is d.u. open.
Therefore it is a fine open set.
\ep

It follows from the lemma that the class of nc functions considered in
\cite{hkm11b,pas13} is what we are calling 
fine holomorphic functions.

\vs
J. Pascoe proved the following inverse function theorem in \cite{pas13}.
The equivalence of (i) and (iii) is due to Helton, Klep and McCullough \cite{hkm11b}.

\bt
\label{thmb3}

Let  $\Omega \subseteq \md$ be an nc domain.
Let $\Phi$ be a fine holomorphic map on $\Omega$. Then the following are equivalent:

(i)  $\Phi$ is injective on $\Omega$.

(ii)
$D\Phi(a)$ is non-singular for every $a \in \Omega$.

(iii) The function $\Phi^{-1}$ exists and is  a  fine holomorphic  map.
\et

\subsection{The Fat topology}
\label{secb2}

Let $\r^+ = \set{r \in \r}{r >0}$. For $n \in \n$, $a  \in \m_n^d$, and $r \in \r^+$, we let $D_n(a,r) \subseteq \m_n^d$ be the matrix polydisc defined by
\be\label{10}
D_n(a,r) = \set{x \in \m_n^d}{\max_{1 \le i \le d}\norm{x_i -a_i} <r}.
\ee
If $a \in \mn^d$, $r \in \r^+$, we define $D(a,r) \subseteq \m^d$ by 
\be\label{20}
D(a,r) = \bigcup_{k=1}^\infty D_{kn}(a^{(k)},r),
\ee
where $a^{(k)}$ denotes the direct sum of $k$ copies of $a$. Finally, if $a \in \m^d$, $r \in \r^+$, we define $F(a,r) \subseteq \m^d$ by
\be\label{30}
F(a,r) = \bigcup_{m=1}^\infty\ \bigcup_{\ u \in \U_m} u^{-1} \big(D(a,r) \cap \m_m^d \big)\ u,
\ee
where $\U_m$ denotes the set of $m \times m$ unitary matrices.
\begin{lem}\label{lem10}
If $a \in \m^d$ and $r \in \r^+$, then $F(a,r)$ is an nc domain.
\end{lem}
\begin{proof}
It is immediate from \eqref{30} that $F(a,r)$ is closed with respect to unitary similarity. To see that $F(a,r)$ is closed with respect to direct sums, assume that $y_1= u_1^{-1}x_1u_1 \in F(a,r)$ and $y_2= u_2^{-1}x_2u_2 \in F(a,r)$ where $x_1, x_2 \in D(a,r)$. Noting that \eqref{10} and \eqref{20} imply that $x_1 \oplus x_2 \in D(a,r)$ we see that
\begin{align*}
y_1 \oplus y_2 &=  (u_1^{-1}x_1u_1)\  \oplus\ ( u_2^{-1}x_2u_2)\\
&= (u_1 \oplus u_2)^{-1} \big(x_1 \oplus x_2\big) (u_1 \oplus u_2)\\
& \in F(a,r).
\end{align*}
\end{proof}
\begin{lem}\label{lem20}
Let $a,b \in \m^d$, $r,s \in \r^+$ and assume that $x \in F(a,r) \cap F(b,s)$. There exists $\epsilon \in \r^+$ such that $F(x,\epsilon) \subseteq F(a,r) \cap F(b,s)$.
\end{lem}
\begin{proof}
Choose $k,l$ and $u,v$ so that
\[
\norm{x -u^{-1}a^{(k)}u} < r\ \ \text{ and }\ \ \norm{x -v^{-1}b^{(l)}v} < s
\]
and define $\epsilon \in \r^+$ by
\[
\epsilon = \min \big\{r-\norm{x -u^{-1}a^{(k)}u} ,\ s-\norm{x -v^{-1}b^{(l)}v} \big\}.
\]

We claim that $F(x,\epsilon) \subseteq F(a,r) \cap F(b,s)$. To prove this claim, fix $y \in F(x,\epsilon)$. By the definition of $F(x,\epsilon)$ there exist $m \in \n$ and a unitary $w$ such that
\be\label{40}
\norm{w^{-1}yw- x^{(m)}} < \epsilon.
\ee
By the definition of $\epsilon$,  $\norm{x -u^{-1}a^{(k)}u} \le r-\epsilon$ so that
\be\label{50}
\norm{x^{(m)} -(u^{(m)})^{-1}a^{(km)}u^{(m)}}  \le  r -\epsilon.
\ee
As \eqref{40} and \eqref{50} imply that $\norm{w^{-1}yw - (u^{(m)})^{-1}a^{(km)}u^{(m)}} < r$ which in turn implies that $y \in F(a,r)$. A similar argument implies that $y \in F(b,s)$.
\end{proof}
Lemma \ref{lem20} guarantees that the sets of the form $D(a,r)$ with $a \in \m^d$ and $r \in \r^+$ form a  basis for a topology on $\m^d$. We refer to this topology as the \emph{fat} topology.

\subsection{The free topology}

The third example of an admissible topology is the 
 {\em free topology}.
A {\em basic free open set } in $\md$ is a set of the form
\[
\gdel \= 
\{ x \in \md : \| \d(x) \| < 1 \},
\]
where $\delta$ is a $J$-by-$J$ matrix with entries in $\pd$.
We define the free topology to be the topology on $\md$ which has as a basis all the sets $\gdel$,
as $J$ ranges over the positive integers, and the entries of $\delta$ range over all polynomials in $\pd$.
(Notice that $ G_{\delta_1} \cap G_{\delta_2} = G_{\delta_1 \oplus \delta_2}$, so these sets do form the basis of a topology). The free topology is a natural topology  when considering semi-algebraic sets.

\begin{prop}\label{prop10}
The fat topology is an admissible topology, finer than the free topology and coarser than the fine topology.
\end{prop}
\bp
All that needs to be shown is that for any $\gdel$ and any $x \in \gdel$, there is a fat neighborhood
of $x$ in $\gdel$. But this is obvious, because $\delta$ is a finite matrix of free polynomials.
\ep

\section{Hessians}
\label{secd}

Let $f$ be an nc function defined on a d.u. open set $U \subseteq \md$, and let $a \in U$.
We define the \emph{Hessian of f at a} to be the bilinear form $Hf(a)$ defined on $\m^d \times \m^d$ by the formula
\[
Hf(a)[h,k] =  \lim_{t \to 0} \frac{Df(a +tk)[h] - Df(a)[h]}{t}, \qquad h,k \in \m^d.
\]
If $A \subseteq \m^d$ and $B \subseteq \m^b$ we define $A\bkt B \subseteq \m^{d+b}$ by
\[
A \bkt B = \bigcup_{n=1}^\infty\  (A \cap \m_n^d) \times (B \cap \m_n^b).
\]
If $\tau$ is a topology on $\m^d$ and $\sigma$ is a topology on $\m^b$, then we let $\tau \bkt \sigma$ be the topology on $\m^{d+b}$ that has a basis
\[
\tau \bkt \sigma = \bigcup \set{A \bkt B}{A \in \tau, B \in \sigma}.
\]
If $\tau$ and $\sigma$ are admissible, then $\tau \bkt \sigma$ is admissible.
\begin{lem}\label{lemd1}
Let $\tau$ be an admissible topology on $\m^d$ and assume that $f:\Omega \to \m^1$ is a $\tau$ holomorphic function. If $\sigma$ is any admissible topology, then $g$ defined on $\Omega \bkt \m^d$ by the formula
\[
g(x,h) = Df(x)[h],\qquad (x,h) \in \Omega \bkt \m^d,
\]
is a $\tau \bkt \sigma $ holomorphic function. Furthermore, for each fixed $n \in \n$ and $x \in \Omega \cap \m_n^d$, $g(x,h)$ is a bounded linear map from $\m_n^d$ to $\m_n^1$.
\end{lem}
\begin{lem}\label{lemd2}
Let $\Omega \subseteq \m^d$ be a fine domain, $f:\Omega \to \m^1$ a fine holomorphic function and $a \in \Omega$. If $h$ and $k$ are sufficiently small, then
\[
f(\begin{bmatrix}
a&k&h&0\\
0&a&0&h\\
0&0&a&k\\
0&0&0&a
\end{bmatrix}) =
\begin{bmatrix}
 f(a)&Df(a)[k]&Df(a)[h]&Hf(a)[h,k]\\
 0&f(a)&0&Df(a)[h]\\
 0&0&f(a)&Df(a)[k]\\
 0&0&0&f(a)
 \end{bmatrix}.
\]
\end{lem}
\begin{proof}
Let
\[
X = \begin{bmatrix} a&k\\0&a\end{bmatrix}
\]
and
\[
H = \begin{bmatrix} h&0\\0&h\end{bmatrix}.
\]
Define a function $g(x,h)$ by
\[
g(x,h) = Df(x)[h], \quad (x,h) \in \Omega \bkt \m^d.
\]
By Lemma \ref{lemd1} $ g$ is a fine holomorphic function of $2d$ variables. Hence, by Lemma \ref{lembb2},
\begin{align*}
g(X,H)&=g(\begin{bmatrix} (a,h)&(k,0)\\0&(a,h)\end{bmatrix})\\
&=\begin{bmatrix} g(a,h)&Dg(a,h)[k,0]\\0&g(a,h)\end{bmatrix}.
\end{align*}
But
\begin{align*}
Dg(a,h)[k,0]&=\lim_{t \to 0} \frac{g(a+tk,h) - g(a,h)}{t}\\
&=\lim_{t \to 0} \frac{Df(a+tk)[h] - Df(a)[h]}{t}\\
&=Hf(a)[h,k].
\end{align*}
Therefore,
\[
g(X,H) = \begin{bmatrix} Df(a)[h]&Hf(a)[h,k]\\0&Df(a)[h]\end{bmatrix}.
\]
Using this last formula and Lemma \ref{lembb2} several times we have that
\begin{align*}
f(\begin{bmatrix}
a&k&h&0\\
0&a&0&h\\
0&0&a&k\\
0&0&0&a
\end{bmatrix}) &= f(\begin{bmatrix}
X&H\\
0&X
\end{bmatrix})\\
&=\begin{bmatrix}
f(X)&g(X,H)\\
0&f(X)
\end{bmatrix}\\
&= \begin{bmatrix}f(a)&Df(a)[k]&Df(a)[h]&Hf(a)[h,k]\\
 0&f(a)&0&Df(a)[h]\\
 0&0&f(a)&Df(a)[k]\\
 0&0&0&f(a)
 \end{bmatrix}.
\end{align*}
\end{proof}

\section{Extending non-singularity  to a fat neighborhood}
\label{sece}

\begin{lem}\label{leme1}
Suppose that $f:U\to \m^1$ is a fat holomorphic function. For each $a \in U$, there exists $r  \in \r^+$ such that $Hf$ is a uniformly bounded bilinear form on $F(a,r)$.
\end{lem}
\begin{proof}
Fix $a \in U$. Since $f$ is a fat holomorphic function, there exists $s,\rho \in \r^+$ such that $F(a,s) \subseteq U$, $f$ is a fine holomorphic function on $F(a,s)$, and
\[
\sup_{x \in F(a,s)} \norm{f(x)} \le \rho.
\]
Let $r = s/2$. If $x \in F(a,r)$, then by the triangle inequality if $\norm{h}, \norm{k} < r/2$, then
\[
\begin{bmatrix}
x&k&h&0\\
0&x&0&h\\
0&0&x&k\\
0&0&0&x
\end{bmatrix} \in F(a,s).
\]
Hence, by Lemma \ref{lemd2},
\[
\norm{Hf(x)[h,k]} \le \rho
\]
whenever  $x \in F(a,r)$ and $\norm{h}, \norm{k} < r/2$.
It follows that if $x \in F(a,r)$, then
\[
\norm{Hf(x)[h,k]} \le \frac{r^2 \rho}{2} \norm{h} \norm{k}
\]
for all $h$ and $k$.
\end{proof}

Now, let $\Omega \subseteq \m^d$ be a fine domain, $f:\Omega \to \m^1$ a fine holomorphic function and $a \in \Omega \cap \m_n^d$.
We set $L = Df(a)$. If $L$ is nonsingular (i.e. surjective), then for each $k\in \n$, ${\rm id}_k \otimes L = Df(a^{(k)})$ is nonsingular as well. Thus, if we set $L_k = {\rm id}_k \otimes L$, then for each $k$, $L_k$ has a right inverse, i.e., a bounded transformation $R:\m_{kn}^1 \to \m_{kn}^d$ such that $L_k R = 1$. 
\begin{defin}\label{defin10}
Let us agree to say that \emph{L is completely nonsingular} if
\[
 \sup_{k} \inf \set{\norm{R}}{R \text{ is a right inverse of }L_k} < \infty.
\]
If $L$ is completely nonsingular, we define $c(L)$ by
\[
c(L) =  \Big(\sup_{k} \inf \set{\norm{R}}{R \text{ is a right inverse of }L_k}\Big)^{-1}
\]
\end{defin}

\begin{lem}\label{lem49}
If $L : \m_n^d \to \m_n^1$ is linear and has a right inverse $R$, then $L$ is completely non-singular
and $c(L) \geq 1/(n \| R \|)$.
\end{lem}
\begin{proof}
Note that ${\rm id}_k \otimes R$ is a right inverse of ${\rm id}_k \otimes L$.
Therefore $c(L)$ is at least the reciprocal of 
\[
\| R \|_{cb} \ := \ \sup_k \|{\rm id}_k \otimes R \|.
\]
By a result of R.~Smith \cite{sm83}; \cite[Prop 8.11]{pau02}, any linear operator $T$ defined on an operator space and 
with range $\m_n$ has 
$\| T \|_{cb} = \| {\rm id}_n \otimes T\| \leq n \| T \|$.
But $R$ is just a $d$-tuple of linear operators from $\m_n$ to $\m_n$, 
so $\| R \|_{cb} \leq n \|R \|$.
\end{proof}

\begin{lem}\label{leme2}
If $L$ is completely nonsingular, $k \in \n$, $E:\m_{kn}^d \to \m_{kn}$ is linear, and $\norm{E} < c(L)$, then $L_k+E$ is nonsingular.
\end{lem}
\begin{proof}Assume that $L$ is completely nonsingular, $k \in \n$, $E:\m_{kn}^d \to \m_{kn}$, and $\norm{E} < c(L)$. Choose $R:\m_{kn}^1 \to \m_{kn}^d$ satisfying $L_kR = 1$ and $\norm{R} \le c(L) ^{-1}$.

If $\norm{E} < c(L)$, then $\norm{ER} < 1$ and as a consequence, $1+ER$ is invertible. But
\begin{align*}
(L_k+E)R(1+ER)^{-1} &= (L_kR +ER)(1+ER)^{-1}\\
&=(1+ER)(1+ER)^{-1}\\
&=1.
\end{align*}
Hence, if $\norm{E} < c(L)$, then $L_k+E$ is surjective.
\end{proof}
\bt\label{prope1}
Let $U \subseteq \m^d$ be a fat nc domain and assume that $f: U \to \m^\ell $ is a fat holomorphic function.
Let $a \in U \cap \mnd$.

(i)
If $Df(a)$ is  full rank, 
then there exists a fat domain $\Omega$ such that $a \in \Omega \subseteq   U$ and $Df(x)$ is full rank
 for all $x \in \Omega$.

(ii)
If $ \ell \leq d$ and $Df(a)$ is an isomorphism from \[
0^{d - \ell} \times \mn^\ell \ : = 
\{ (0, \dots, 0, h^{d - \ell + 1}, \dots, h^d) : h^r \in \mn, d- \ell + 1 \leq r \leq d \}
\]
onto $\mn^\ell$, then there is  a fat domain $\Omega$ such that $a \in \Omega \subseteq   U$ and $Df(x)$ is nonsingular on 
$ 0^{d - \ell} \times \m_\mu^{\ell}$ 
 for all $\mu \in \n$ and for all $x \in \Omega \cap \m_\mu^d$.
\et
\bp
Let $f = (f^1, \dots, f^\ell)^t$.
By Lemma \ref{leme1}, there exist $s, M \in \r^+$ such that, for each $1 \leq j \leq \ell$,
\[
\norm{Hf^j(x)[h,k]} \le M\norm{h}\norm{k}
\]
for all $x \in F(a,s)$ and all $h,k \in \m^d$ that have the same size as $x$. Choose $r \in \r^+$ satisfying
\[
r < \min \big\{s, \frac{c(Df(a))}{ M \sqrt{\ell} }\big\}.
\]

Let $m \in \n$ and $x \in F(a,r) \cap \m_{mn}^d$ (so that $\norm{x - a^{(m)}} < r$). We have that
for each $j$
\begin{align*}
\norm{Df^j(x)[h] - Df^j(a^{(m)})[h]}&=\norm{\int_0^1 \frac{d}{dt}Df^j\big(a^{(m)}+t(x-a^{(m)})\big)[h]dt}\\
&=\norm{\int_0^1 Hf^j\big(a^{(m)}+t(x-a^{(m)})\big)[h,x-a^{(m)}]dt}\\
&\le M \norm{h} \norm {x - a^{(m)}}\\
&< \frac{\,  c(Df(a))}{ \sqrt{\ell}  }  \norm{h}      .
\end{align*}
So
\[
\norm{Df(x) - Df(a^{(m)})} <  c(Df(a)).
\]
Hence, by Lemma \ref{leme2}, $Df(x)$ is nonsingular, proving (i).

Part (ii) follows in the same way, by considering $Df(x) |_{0^{d-\ell} \times \m_m^\ell}$.
By hypothesis, this has a right inverse at $a$, so by Lemma~\ref{lem49} is completely 
nonsingular. 
Therefore there is a fat neighborhood of $a$ (perhaps smaller than in case (i)) on which 
 $Df(x) |_{0^{d-\ell} \times \m_m^\ell}$ is nonsingular.
\ep

We can now prove a fat version of the inverse function theorem, Theorem~\ref{thmb3}.

\bt
\label{thme1}
Let  $\Omega \subseteq \md$ be a fat nc domain.
Let $\Phi$ be a fat holomorphic map on $\Omega$. Then the following are equivalent:

(i)  $\Phi$ is injective on $\Omega$.

(ii)
$D\Phi(a)$ is non-singular for every $a \in \Omega$.

(iii) The function $\Phi^{-1}$ exists and is  a  fat holomorphic  map.
\et
\bp
In light of Pascoe's Theorem \ref{thmb3}, all that remains to prove is that Asumption (ii) implies that $\Phi^{-1}$ is fat holomorphic.
Let $U = \Phi(\O)$, and let $b = \Phi(a) \in U \cap \mnd$.
We must find a fat neighborhood of $b$ on which $\Phi^{-1}$ is bounded.
This in turn will follow if we can find $r,s > 0$ such that
\be
\label{eqe5}
\Phi( D(a,r) ) \supseteq D(b,s),
\ee
 where $D(a,r)$ is defined in \eqref{20}.
By Lemma~\ref{leme1}, there exists $r_1 > 0, M$ such that the Hessian of
$f$ is bounded by $M$ on $F(a,r_1)$.
Choose $0 < r < r_1$ so that 
\[
M r \ < \ \frac{1}{2} c( D \Phi(a) ) ,
\]
and choose $ s > 0$ so that
\[
s \ < \ 
 \frac{r}{2} c( D \Phi(a) ) .
\]
We claim that with these choices, \eqref{eqe5} holds.

Indeed, choose $k \in \n$, and let $x \in D_{kn} (\ak, r) $.
Let us write $\a$ for $\ak$.
Then
\beq
\| \Phi (x) - \Phi(\a) \|
&\=&
\| \int_0^1 \frac{d}{dt} \Phi( \a + t (x - \a) ) dt \|
\\
&\=&
\| \int_0^1 D\Phi( \a + t (x - \a) )[x - \a] dt \|
\\
&\=&
\| D \Phi (\a) [ x - \a] + \\
&&
\int_0^1 D\Phi( \a + t (x - \a) )[x - \a] - D \Phi(\a) [ x - \a] dt \|
\\
&\geq&
\| D \Phi (\a) [ x - \a]  \| - M \| x - \a \|^2 \\
&\geq&
\left( c(D \Phi(a)) - M \| x - \a \| \right) \| x - \a \| \\
&\geq&
\frac{1}{2}  c(D \Phi(a)) \| x - \a \| .
\eeq
Since $D\Phi$ is non-singular, we have that $\Phi(D_{kn} (\ak, r) )$ is an open connected set, and
by the last inequality it contains $D_{kn}( b^{(k)} , s)$.
\ep

\section{The implicit function theorem}
\label{secc}

Let $f = (f_1 , \dots, f_k)$ be an
$\L(\c,\c^k)$ valued nc function.
We shall let $Z_f = \cap_{i=1}^k Z_{f_i}$ denote the zero set of $f$.
If $a \in \mnd$, the derivative of $f$ at $a$, $Df(a)$, is a linear map
from $\mnd$ to $\mnk$. We shall say that $Df(a)$ is of {\em full rank} if the rank of this
linear map is $kn^2$.

For convenience in the following theorem, we shall write 
$h$ in $\mnk$ as $h= (h^{d-k+1}, \dots, h^d)$.

\bt
\label{thmc1}
Let $U$ an nc domain.
Let $f$ be an
$\L(\c,\c^k)$ valued fine holomorphic function on $U$, for some $1 \leq k \leq d-1$.
Suppose
\begin{eqnarray}
\nonumber
\lefteqn{
\forall\, n \in \n,\
\forall \, a \in U \cap \mnd, 
}
 & \\
&\forall \, h 
\in \mnk \setminus \{ 0 \}, \quad
Df(a) [(0,\dots,0,h^{d-k+1},\dots,h^d)] \neq 0.
\label{eqc1}
\end{eqnarray}
Let $W$ be the projection onto the first $d-k$ coordinates of $Z_f \cap U$.
Then there is an $\L(\c,\c^k)$-valued fine holomorphic  function $g$ on $W$ such that
\[
Z_f \cap U \=
\{ (y,g(y)) \, : \, y \in W \} .
\]

Moreover, if $f$ is fat holomorphic, then $g$ can also be taken to be fat holomorphic.
\et
\bp
Let $\Phi(x) = (x^1, \dots, x^{d-k}, f(x) )^t$ be the nc map defined on $U$
by prepending the first $d-k$ coordinate functions.
By \eqref{eqc1}, $\Phi$ is non-singular on $U$, so by Theorem~\ref{thmb3},
there is an nc map from $U$ onto some  set $\Omega$, with inverse $\Psi$.

Let us write points $x$ in $\mnd$ as $(y,z)$, where $y \in \m_n^{d-k}$ and $z \in \mnk$.
Then  $y$ is in $W$ iff there is some $z$ such that $(y,z) \in U$ and $f(y,z) = 0$.

Let $\Psi = \psi_1 \oplus \psi_2$, where $\psi_1$ is $\Psi$ followed by projection onto the first
$d-k$ coordinates, and 
$\psi_2$ is $\Psi$ followed by projection onto the last
$k$ coordinates. Define $g(y) = \psi_2(y^1,\dots, y^{d-k}, 0, \dots, 0 )$.

If $(y,z) \in Z_f \cap U$, then $\Phi(y,z) = (y,0)$ and 
\[
\Psi \circ \Phi (y,z) \= (y, z) \= (\psi_1( y,0), g(y) ) ,
\]
so $z = g(y)$.

Conversely, if $y \in W$ and $z = g(y)$, then 
$\Psi(y,0) = (\psi_1(y,0), g(y) )$, so
\[
\Phi \circ \Psi (y, 0) \= (y,0) \= (\psi_1(y,0), f(\psi_1(y,0),g(y)) .
\]
Therefore $f(y,g(y)) = 0$.

Finally, if $f$ is fat holomorphic, then by Theorem~\ref{thmc1} the function $\Psi$ is fat holomorphic,
and hence so is $g$.
\ep

Two questions naturally arise. 
The first is whether satisfying \eqref{eqc1} at a particular point automatically
leads to it holding on a neighborhood. Theorem~\ref{prope1} shows that this is true in the
fat category.

The second question is whether whenever $Df(a)$ is of full rank,
one can change basis to obtain condition \eqref{eqc1}. We shall show in Corollary~\ref{corc1}
that the answer generically is yes.

%
%

\bd
\label{defbd1}
Let $ d \geq 2$.  We shall say that a $d$-tuple $x \in \mnd$ is {\em broad}
if \[
\{ p(x) : p \in \pd \} \= \mn .
\]
\ed

\bt
\label{thmc2}
Let $d \geq 2$, let $a \in \mnd$, and assume  $a$ is broad.
Let $N\leq (d-1)n^2 +1$.
Suppose $H_1, \dots, H_N \in \mnd$ are linearly independent modulo
$\{ a \Gamma - \Gamma a : \Gamma \in \mn \}$.
Then, for every $K_1, \dots, K_{N} \in \mn$ and for every $M \in \mn$ 
there exists $p \in \pd$ such that 
\be
\label{eqc2}
p(a) \= M,\quad {\rm and}\quad \ Dp(a)[H_i] = K_i, \forall\,  i \leq N .\ee
\et

\bp
We shall prove the theorem by induction on $N$. When $N=0$, the conclusion holds
because $a$ is broad. So assume that the theorem has been proved for some $ 0 \leq N \leq 
(d-1)n^2$, and we wish to show the conclusion holds for $N+1$. 
Fix $H_1, \dots, H_{N+1}$.
Assume that 
\be
\label{eqc2.1}
H_{N+1}
\ \notin \ 
\{ a \Gamma - \Gamma a: \Gamma \in \mn\} + \vee \{ H_1, \dots, H_N \} .
\ee

Let 
\[
I \= \{ p \in \pd : p(a) = 0,  Dp(a) [ H_i] = 0, i \leq N \} .
\]

Case 1: $N \geq 1$, and for all  $p \in I$, we have  $Dp(a)[H_{N+1}] = 0$.

If this holds, then by Lemma~\ref{lembb2} the map
\[
\pi: \ 
p(
\beginb
\beginb
a& H_1 \\
0 & a
\endb
&&0\\
&\ddots&\\
0&&
\beginb
a& H_N \\
0 & a
\endb
\endb
)
\
\mapsto\
p(
\beginb
a& H_{N+1} \\
0 & a
\endb
)
\]
is a well-defined homomorphism, as $p$ ranges over $\pd$.
By the inductive hypothesis and Lemma \ref{lembb2}, we have that for all
$K = (K_1,\dots, K_N)$, 
\[
\pi: \
\beginb
\beginb
M& K_1 \\
0 & M
\endb
&&0\\
&\ddots&\\
0&&
\beginb
M& K_N \\
0 & M
\endb
\endb
\ \mapsto \
\beginb
M& L(M,K) \\
0 & M
\endb
\]
for some linear map $L$.
Letting $K = 0$ and using the fact that $\pi$ is multiplicative, we get
\[
M_1 L(M_2, 0) + L(M_1,0) M_2 \= L(M_1M_2,0) .
\]
This means that the map $M \mapsto L(M,0)$ is a derivation on $\mn$, so it must
be inner \cite[Thm 3.22]{fd93}. Therefore there exists $\Gamma \in \mn$ such that
\be
\label{eqc3}
L(M,0) \= M \Gamma - \Gamma M .
\ee
As 
\be
\label{eqc3.1}
\beginb
\beginb
M& K_1 \\
0 & M
\endb
&&0\\
&\ddots&\\
0&&
\beginb
M& K_N \\
0 & M
\endb
\endb
\
\beginb
\beginb
0& K_1 \\
0 & 0
\endb
&&0\\
&\ddots&\\
0&&
\beginb
0& K_N \\
0 & 0
\endb
\endb
\ee
on the one hand maps to 
\[
\beginb
M& L(M,K) \\
0 & M
\endb
\
\beginb
0& L(0,K) \\
0 & 0
\endb
\]
and on the other to
\[
\beginb
0& L(0,MK) \\
0 & 0
\endb,
\]
we conclude
\be
\label{eqc4}
L(0,MK) \= M L(0,K) ,
\ee
and by reversing the factors in \eqref{eqc3.1} get
\be
\label{eqc5}
L(0,KM) \= L(0,K) M .
\ee
Let $E_i \in \mn^N$ have the identity in the $i^{\rm th}$ slot, and $0$ elsewhere.
By \eqref{eqc4} and \eqref{eqc5}, we have
$L(0,E_i)$ commutes with every matrix in $\mn$, so must be a scalar.
By linearity and \eqref{eqc4} again, we get that 
\be
\label{eqc6}
L(0,K) \= \sum_{i=1}^N c_i K_i .
\ee
As $\pi$ is linear, we have $L(M,K) = L(M,0) + L(0,K)$, so
combining  this observation  with \eqref{eqc3} and \eqref{eqc6}, we conclude that
\be
\label{eqc7}
L(M,K) \= M \Gamma - \Gamma M + \sum_{i=1}^N c_i K_i .
\ee
By Lemma \ref{lembb2}, this means
\be
\label{eqc8}
Dp(a) [H_{N+1}] \= a \Gamma - \Gamma a +  \sum_{i=1}^N c_i Dp(a) [H_i] .
\ee
Let $p(x) = x^r$, the $r^{\rm th}$ coordinate function, in \eqref{eqc8}.
This yields
\be
\label{eqc9}
H^r_{N+1} \=  a \Gamma - \Gamma a +  \sum_{i=1}^N c_i H^r_i .
\ee
As \eqref{eqc9} holds for $1 \leq r \leq d$ with the same $\Gamma$, this contradicts \eqref{eqc2.1}.

\vs
Case 2:  $N = 0$, and for all  $p \in I$, we have  $Dp(a)[H_{1}] = 0$.

Now the inductive hypothesis is that for all $M \in \mn$, there is a polynomial $p$ with $p(a) = M$.
The ideal $I$ is all polynomials that vanish at $a$. As in Case 1, we conclude that the map
\[
\pi : p(a) \ \to \
p(
\beginb
a& H_{1} \\
0 & a
\endb
) \=
\beginb
p(a)&Dp(a)[ H_{1}] \\
0 &p( a)
\endb
\]
is a well-defined homomorphism, and that
\[
Dp(a)[ H_{1}]
\]
is a derivation on $\{ p(a) \}$, so 
\[
Dp(a)[ H_{1}] \=
p(a) \Gamma - \Gamma p(a)
\]
for some $\Gamma \in \mn$.
Letting $p$ be each of the coordinate functions in turn, we get $H_1 = a \Gamma - \Gamma a $,
a 
contradiction to \eqref{eqc2.1}.

\vs
Case 3: As the previous two cases have been ruled out, we must be in the situation that for some
$p \in I$,  $Dp(a)[H_{N+1}] \neq 0 $.
As
\[
D qp (a) [H] \= Dq(a) [H] p(a) + q(a) Dp(a) [H] ,\]
we have that 
\[
{\cal D} \ :=\ \{ Dp(a)[H_{N+1}] : p \in I \}
\]
is invariant under multiplication on the left or right by elements of
\[
\{ q(a) : q \in I \} .
\]
Since $a$ is broad, we have that ${\cal D}$ is a non-empty ideal in $\mn$, and therefore all of $\mn$.

Choose now $M$ and  $K_1, \dots, K_{N+1}$ in $\mn$. By the inductive hypothesis, we can find 
a polynomial $q$ such that
\[
q(a) \= M,\quad {\rm and}\quad \ Dq(a)[H_i] = K_i, \forall\,  i \leq N .
\]
Since we are in Case 3, there is a polynomial $p \in I$ such that
\[
Dp(a) [H_{N+1}] \= 
K_{N+1} - Dq(a) [H_{N+1}] .
\]
Then the polynomial $r = p+q$ satisfies
\[
r(a) \= M,\quad {\rm and}\quad \ Dr(a)[H_i] = K_i, \forall\,  i \leq N+1 .
\]
\ep

As a consequence,  if
$Df(a)$ is of full rank, then, generically, there is a polynomial change of variables that
allows one to assume it is of full rank on $0^{d-k} \oplus \mnk := \{ (0, \dots, 0 , h) : h \in \mnk \}$.

\begin{cor}
\label{corc1}
Let $d \geq 2$, let $\Omega \subseteq \md$ be an nc domain, and 
fix $1 \leq k \leq d-1$.
Let $f$ be an
$\L(\c,\c^k)$ valued fine holomorphic function on $\Omega$.
Suppose that  $Df(a)$ is of  rank $kn^2$ for some point $a \in \Omega \cap \mnd$.
Suppose also that $(a^1, \dots, a^{d-k}, f(a))$ is broad, and 
the commutant of $a$ is $\c$.

Then there are 
a d.u open set $U$ containing a broad point $b$ and
an invertible nc polynomial map $\Phi$ from
 $U$ into $\Omega$, mapping the point $b$ to $a$, such that,
\begin{eqnarray}
\nonumber
&\forall \, h = (h^{d-k+1}, \dots, h^d)\in \mnk \setminus \{ 0 \} \quad\\
&\qquad D f \circ \Phi (b) [(0,\dots,0,h^{d-k+1},\dots,h^d)] \neq 0.
\label{eqc20}
\end{eqnarray}

Moreover, if $\Omega$ is fat, then $U$ can be chosen to be a fat nc domain.
\end{cor}
\bp
%

Choose $b = (a^1, \dots, a^{d-k}, f(a))$. 
By the chain rule, \eqref{eqc20} will hold provided 
\be
\label{eqc21}
\left\{D \Phi(b) [ \{ 0^{d-k} \oplus \mnk \} ] \right\}\ \cap\ 
{\rm ker} Df(a) \= \{ 0 \} .
\ee
By Theorem~\ref{thmc2} we can choose the polynomial entries $(p^1, \dots, p^d)$ of
$\Phi$ so that $\Phi(b) = a$ and 
the action of the derivative is arbitrary, except on the set
$
\{ b \Gamma - \Gamma b \} 
 $.
But on this set, by Lemma~\ref{lembb3}, we have
\be
\label{eqc20.1}
D\Phi (b) [ b \Gamma - \Gamma b ]
\=
\Phi(b)  \Gamma - \Gamma \Phi(b) \= a \Gamma - \Gamma a .
\ee
If this were in the kernel of $Df(a)$, we would have
\[
0 \=
Df(a) [a \Gamma - \Gamma a ] = f(a) \Gamma - \Gamma f(a) .
\]
But if this holds, and $b \Gamma - \Gamma b $ is in $\{ 0^{d-k} \oplus \mnk \}$,
then $b \Gamma - \Gamma b =0$. 

So  for any choice of $\Phi$ with $\Phi(b) = a$,
we have 
\[
\left\{D\Phi (b) [ \{ b \Gamma - \Gamma b\} \cap   \{ 0^{} \oplus \mnk \} ] \right\}
\ \cap \
{\rm ker} Df(a) \= \{ 0 \}
.\]
As $b$ is broad and $\{a\}' = \c$, the sets
$\{ b \Gamma - \Gamma b \}$ and $\{ a \Gamma - \Gamma a \}$ are both of dimension
$n^2 - 1$.
Now choose the derivatives of $\Phi$ in a set of directions that complements
$ \{ b \Gamma - \Gamma b \} $ so that $D\Phi(b)$ is of full rank and
\eqref{eqc21} holds.
Let
\be
\label{eqc22}
U \= \Phi^{-1} (\Omega) \cap \{ x: D \Phi (x) {\rm \ is\ invertible} \} .
\ee

Finally, if $\Omega$ is fat, then choose $U$ to be the intersection of the fat nc domain
$\Phi^{-1} (\Omega) $ with a fat neighborhood of $b$ on which $D \Phi$ is invertible,
which exists by Theorem \ref{prope1}.
\ep


%
%

\section{The range of an nc function}
\label{secf}

A necessary and sufficient condition that
the function
\[
f : s^{-1} x s \ \mapsto \ s^{-1} z s 
\]
is well-defined on the similarity orbit $S_{x}$ of $x$
is that $z$ be in $\{ x\}''$.
So if $f$ is an nc function on a d.u. open set, then for every $M$ in the commutant of $x$,
 $1 + tM$ must commute with $f(x)$ for $t$ small.
This imposes  the requirement that
\[
f(x) \in \{ x \}'' .
\]
When $d = 1$, we have $\Ax =  \{ x \}''$, but this containment can be proper for $d > 1$.
(By $\Ax$ we mean the algebra generated by $x$).
\begin{ques}
\label{r1}
If $f$ is a $\tau$ nc function on an $\tau$ open set $U$, is $f (x) \in \Ax$?
\end{ques}

A necessary condition for $f$ to be pointwise approximable by polynomials 
is that $f (x) \in \Ax$. In \cite{amfree}, the authors proved that a free holomorphic function
is locally the uniform limit of free polynomials, so the answer to Question~\ref{r1}
is yes for the free topology.

We shall show that the answer is no for the fat (and hence for the fine) topology.

 Indeed, let $x_0 \in \m_2^2$ be 
\[
x_0 \= 
\left[
\begin{pmatrix} 
0&1\\
0&0
\end{pmatrix} ,
\
\begin{pmatrix} 
1&0\\
0&0
\end{pmatrix}  \right],
\]
and let $z_0 \in \m^2 $ be
\[
z_0 \= 
\begin{pmatrix} 
0&0\\
1&0
\end{pmatrix}.
\]

As $\{ x_0 \}'$ is just the scalars, 
we have $z_0 \in \{ x_0 \}'' \setminus \A_{x_0}$, and
the function
\be
\label{r2}
f : s^{-1} x_0 s \ \mapsto \ s^{-1} z_0 s 
\ee
is well-defined on the similarity orbit $S_{x_0}$ of $x_0$.
We shall show that it extends to a fat holomorphic function.

%
%
%
%
Define $p$ by
\be
\label{eqr13}
p(X,Y,Z) \= (Z)^2 + XZ + ZX +YZ - \idd.
\ee
If $x_0 = (X,Y)$ and $z_0 = Z$ are substituted in \eqref{eqr13},
we get $p(x_0,z_0) = 0$.
Let 
\be
\label{eqr14}
a \= 
\left[
\begin{pmatrix} 
0&1\\
0&0
\end{pmatrix} ,
\
\begin{pmatrix} 
1&0\\
0&0
\end{pmatrix}
\
\begin{pmatrix} 
0&0\\
1&0
\end{pmatrix}
\right],
\ee

\begin{lem}
\label{lemr14}
\be
\label{eqr16}
\frac{\partial}{\partial Z} p(a) [h] \=
\begin{pmatrix} 
h_{11} + h_{12} + h_{21} & h_{11} + h_{12} + h_{22} \\
h_{11} + h_{22}   &    h_{12} + h_{21}
\end{pmatrix} .
\ee
\end{lem}

It is immediate from \eqref{eqr16} that $ \frac{\partial}{\partial Z} p(a): \m_2 \to \m_2$ is onto, and so has a 
right inverse. 
By Theorem~\ref{prope1},
there is a fat domain $\Omega \ni a$
such that $\frac{\partial}{\partial Z} p(\lambda)$ is non-singular for all $\lambda \in \Omega$.

Now we invoke Theorem~\ref{thmc1}. Let $V$ be the projection onto the first two coordinates
of $\Omega$. This is a fat domain containing $x_0$.
We conclude:
\bt
\label{thmf1}
There is a fat domain $V$ containing $x_0$ and a fat holomorphic function $g$ defined on $V$ 
such that $g(x_0) \notin \A_{x_0}$.
\et

\section{No free implicit function theorem}
\label{sech}

In this section we prove that the 
implicit function theorem \ref{thmc1}
is false in the free category. Indeed, we show that there is a dichotomy: one cannot have an admissible topology $\tau$  for which the maps $x  \mapsto \| q (x)\|$ are continuous for all $q \in \pd$
and for which one has both an implicit function theorem (as in the fat and fine topologies) and
an affirmative answer to Question~\ref{r1}.

Let $p(X,Y,Z)$ be as in \eqref{eqr13}, and define
\be
\label{eqh03}
\Phi(X,Y,Z) = (X,Y,p(X,Y,Z)).
\ee

Recall the following condition on solving a Sylvester equation, also called
 a matrix Ricatti equation
\cite{syl}.
\bl
\label{lemh05}
The matrix equation
$AH-HB =0$, for $A,B,H \in \mn$, has a non-zero solution $H$ if and only if
$\sigma(A) \cap \sigma(B) \neq \emptyset.$ The dimension of the set of solutions is
$ \# \{ (\lambda,\mu) :  \lambda \in \sigma(A), \mu \in \sigma(B) , \lambda = \mu  \} $,
where eigenvalues are counted with multiplicity.
\el

\bl
\label{lemh1}
The derivative of $\Phi$, and  $\frac{\partial}{\partial Z} p$,  are each non-singular if and only if
\be
\label{eqh1}
\sigma(X + Y + Z) \cap \sigma (-X-Z) \= \emptyset.
\ee
\el
\bp
$D\Phi$ is non-singular if and only if $\frac{\partial}{\partial Z} p$ is.
\beq
\frac{\partial}{\partial Z} p (X,Y,Z)[H]
&\=&
ZH + HZ + XH + HX + Y H \\
&=&
(X+Y+Z) H + (X+Z) H .
\eeq
The result now follows from Lemma \ref{lemh05}.
\ep
\bl
\label{lemh2}
Let $a$ be as in \eqref{eqr14}.
There is a free neighborhood of $a$ on which \eqref{eqh1} holds; moreover it is of the form
$\gdel$ where $\delta$ is a diagonal matrix of polynomials.
\el
\bp
The eigenvalues of $a^1 + a^2 + a^3$ are $(1 \pm \sqrt5)/2$; call them $\l_1 $ and $\l_2$.
The eigenvalues of $a^1 + a^3$ are $\pm 1$.
Let $\vare >0$ be such that the closed disks of radius $\vare$ and centers
$\l_1,\l_2,1,-1$ are disjoint.

Let $\delta(x)$ be the $2$-by-$2$ diagonal matrix with entries
\[
M(x^1 + x^2 + x^3 - \lambda_1) (x^1 + x^2 + x^3 - \lambda_2) \ {\rm and\ }
M(x^1 + x^3 - 1)(x^1 + x^3 + 1).
\]
By choosing $M$ large enough, one can ensure that if $x \in \gdel$, then
\[
\sigma(x^1 + x^2 + x^3) \subset \D(\l_1,\vare) \cup \D(\l_2,\vare)\
{\rm and \ }
\sigma(x^1  + x^3) \subset \D(1,\vare) \cup \D(-1,\vare).
\]
\ep
\bt
\label{thmh1}
Let $\tau$ be an admissible topology, defined on $\md$ for all $d \geq 2$.
Suppose $\tau$ has the property that for each $q \in \pd$,
the map $x \mapsto ||q(x)||$ is $\tau$-continuous from $\md$ to $\R^+$.
If every $\tau$ holomorphic function is pointwise approximable by free polynomials, then 
Theorem \ref{thmc1} does not hold in the $\tau$ category.

If, in addition, $\tau$ has the property that   the projection maps from $\md$ to $\m^{[d-1]}$ are open, then
Theorem~\ref{thmb3} also does not hold in the $\tau$ category.
\et
\bp
Let $\Phi$ be as in \eqref{eqh03} and $\gdel$ as in Lemma~\ref{lemh2}.
By Lemma~\ref{lemh1},  $\frac{\partial}{\partial Z} p$ and   $D\Phi$ are non-singular on $\gdel$,
and by hypothesis, $\gdel$ is $\tau$-open.
If the  Implicit function theorem were true for $\tau$, applying it to the set  $Z_p \cap \gdel$, 
there would be a $\tau$ open neighborhood $W$ of $x_0 \in \m_2^2$ and a $\tau$ holomorphic
function $g$ such that $g(x_0) = z_0$. This cannot occur, because $z_0 \notin \A_{x_0}$.

If the $\tau$ Inverse function theorem were true, applying it to the map $\Phi$ on $\gdel$ 
and repeating the proof of Theorem~\ref{thmc1} would yield the $\tau$ Implicit function theorem and 
the function $g$.
\ep

\begin{cor}
Theorem \ref{thmc1} does not hold in the free category.
\end{cor}


\section{Free Algebraic Sets}
\label{secg}

By a {\em free algebraic set} in $\md$ we mean the common zero set of some set of free polynomials.

\begin{example}
\label{exg1}

Consider the polynomial
\[
p(X,Y) \= a X^2 + b XY + c YX ,
\]
where $b \neq - c$, and let $V = Z_p$.
The partial derivative with respect to $Y$ is
\[
\frac{\partial}{\partial Y} p(X, Y) [ H] \= b  X H + c H X .
\]
By Lemma~\ref{lemh05}, the Sylvester equation $
b  X H + c H X = 0 $
has a non-zero solution if and only if $\sigma (bX) \cap \sigma (-cX) $ is non-empty.
Assume that 
\be
\label{eqg1}
p(X_0,Y_0) = 0 , \quad {\rm and }\quad  \sigma (bX_0) \cap \sigma (-cX_0) = \emptyset .
\ee
Then there is a fat neighborhood of $X_0$ on which 
\be
\label{eqg2}
\sigma (bX) \cap \sigma (-cX)  \= \emptyset ,
\ee
so by Theorem~\ref{thmc1} there is a function $g$ such that locally 
$V = \{ (X, g(X) ) \}$.
In particular, this forces $Y$ to commute with $X$, so locally
\[
X( aX+ (b+c)Y) \= 0.
\]
Therefore 
\be
\label{eqg5}
Y \= -\frac{a}{b+c} X
\ee
since  $X$ is invertible by \eqref{eqg2}.

So if \eqref{eqg1} holds, $X_0$ and $Y_0$ commute, and 
\[
Y_0 \= -\frac{a}{b+c} X_0 .
\]

Dropping assumption \eqref{eqg2}, how many non-commuting solutions are there?
For example, the non-commuting pair
\[
\left[
\begin{pmatrix} 
b&0\\
0&-c
\end{pmatrix} ,
\
\begin{pmatrix} 
-\frac{ab}{b+c}&0\\
e& \frac{ac}{b+c}
\end{pmatrix}  
\right]
\]
satisfies $p(X,Y) =0$ for any $e \in \C$.

Let $k$ be the number of common eigenvalues of $bX$ and $-cX$, counting multiplicity.
%
For fixed  $X$, the equation \[
bXY - c YX \= - aX^2
\]
 always has one solution given by \eqref{eqg5}.
By Lemma~\ref{lemh05}, it therefore has a $k$ dimensional set of solutions.
If $X$ is invertible, the solution from \eqref{eqg5} is the unique commuting one, so all the others
do not commute.

What is the dimension of the set of non-commuting pairs $(X,Y)$ in $\mn^2$ annihilated by $p$?
If $-b/c$ is a root if unity, it can be larger than $n^2$. But if $\alpha = -b/c$ is not a root of unity,
it is exactly $n^2$ when $n \geq 2$.
Indeed, suppose $X$ has eigenvalues
\[
\l_1, \alpha \l_1, \dots, \alpha^{k_1} \l_1, \l_2, \alpha \l_2, \dots, \alpha^{k_2} \l_2, \dots,
\l_r, \dots, \alpha^{k_r} \l_r 
\]
with corresponding multiplicities
\[
d_{1,0}, d_{1,1}, \dots, d_{1, k_1},d_{2,0}, d_{2,1}, \dots , d_{2,k_2},\dots, d_{r,0}, \dots, d_{r, k_r} ,
\]
where for $i \neq j$, $\l_i$ is not a power of $\alpha$ times $\l_j$.
Then $$
k \= \sum_{i=1}^r \sum_{j=1}^{k_i} d_{i,j-1} d_{i,j} .
$$
The dimension of the set of $X$'s with this collection of eigenvalues is
\be
\label{eqg8}
r + n^2 - \sum_{i=1}^r \sum_{j=0}^{k_i} d_{i,j}^2 .
\ee
As 
\[
 \sum_{j=1}^{k_i} d_{i,j-1} d_{i,j} +1\  \leq \ \sum_{j=0}^{k_i} d_{i,j}^2 ,
\]
we get that the dimension of the set of pairs $(X,Y)$ in $Z_p$, 
which is \eqref{eqg8} plus $k$, is at most $n^2$.
However, this is attained with $ k  > 0$ by, for example, choosing $d_{1,0} = d_{1,1} = 1$,
and for $i > 1$, choosing $d_{i,0} = 1$ and $d_{i,j} = 0, j \geq 1$.

We summarize:

%

\bprop
\label{propg3} Assume $b \neq - c$. Let $X_0 \in \mn$ be fixed and invertible.
Let 
\[
{\cal Y} \= \{ Y \in \mn : 
a X_0^2 + bX_0Y + cYX_0 = 0 \} .
\]
Let $k$ be the number of common eigenvalues of $b X_0$ and $-cX_0$, counting multiplicity.

(i) If $k=0$, then ${\cal Y}$ has a unique element, which commutes with $X_0$.

(ii) If $k > 0$, then ${\cal Y}$ is a $k$-dimensional affine space in $\mn$,  and it contains a unique
element that commutes with $X_0$.

(iii) If $b/c$ is not a root of unity, then the dimension  of the set of non-commuting solutions 
in $\mn^2$ of $p(X,Y) = 0$ is exactly $n^2$ if $n \geq 2$, the same as the dimension of the set of commuting solutions.
\eprop

\end{example}
%
%
The example 
$$
p(X,Y) = (XY-YX)^2 - \idd
$$ shows that one can choose a polynomial
for which $Z_p$ contains no commuting elements, 
but  for a generic $p$ this does not happen.
We can extend this observation to ``codimension one''
free algebraic sets. For convenience, let us write elements of $\md$ as
$(X,Y^1, \dots, Y^{d-1})$.
\bt
\label{thmg1} Let $k=d-1$, and let $p_1, \dots, p_k$ be free polynomials in $\pd$
with the property that, when evaluated on $d$-tuples of complex numbers, they are not
constant in the last $k$ variables. Let $p = (p_1, \dots, p_k)^t$, and let 
\[V \= \{ (X,Y^1, \dots, Y^{k})\  : \ p(X, Y^1, \dots, Y^{k}) = 0\}.
\]

Let $B$ be the finite (possibly empty) set
\[
B \= \cup_{j=1}^k \{ x \in \C \ : \  \forall y \in \c^k, \ \ p_j(x,y^1, \dots, y^k) \neq 0\} .
\]
If $X_0$  in $\mn$ has $n$ linearly independent eigenvectors and $\sigma(X_0) \cap B = \emptyset$,
then there exists $Y_0$ in $\mn^k$ that  satisfies
$(X_0,Y_0) \in V$ and such that each element $Y_0^j$ commutes with $X_0$.

(ii) If $(X_0,Y_0) $ is in $V$ and $X_0$ and $Y_0$ do not commute, then we must have
\be
\label{eqg4}
(X_0,Y_0) \ \in \ V \cap 
\{ (X,Y) : D p (X,Y) {\rm \ is\ not\ full\ rank\ on\ } 0 \times \mn^k \}.
\ee
%
%
\et

\bp
%

(i) Write  $X_0$ as the diagonal matrix with diagonal entries $(x_1, \dots, x_n)$ with respect to a
basis of eigenvectors. 
Choose $Y_0^j$ to be  the diagonal matrix with diagonal entries $(y_1^j, \dots, y^j_n)$. Then $Y_0^j$ will commute with $X_0$; and $p(X_0,Y_0)$ will be zero if
$p(x_i,y^1_i, \dots, y^k_i) = 0$ for each $i$. This can be done by choosing $y_i$ to be a root of the polynomial
$p(x_i,y)$.

(ii) By Theorems~\ref{prope1} and \ref{thmc1}, if $ Dp$ is full rank on $0 \times \mnk$, then
 there is a fat holomorphic function $g$ that maps
$X_0$ to $Y_0$. Since $g$ is a function of one variable, this means $Y_0^j$ is in $\A_{X_0}$ for each $j$, and so
commutes with $X_0$. 
\ep

%

\bibliography{../../references}

\end{document}